\documentclass[1 [leqno,11pt]{amsart}
\usepackage{amssymb, amsmath,amsmath,latexsym,amssymb,amsfonts,amsbsy, amsthm}
\def\inte#1{
\displaystyle\mathop{#1\kern0pt}^\circ }

\newcommand{\beqo}{\begin{equation*}}
\newcommand{\eeqo}{\end{equation*}}
\newcommand{\beno}{\begin{eqnarray*}}
\newcommand{\eeno}{\end{eqnarray*}}
\numberwithin{equation}{section}
\let\pa=\partial
\let\al=\alpha
\let\b=\beta
\let\g=\gamma

\let\e=\varepsilon

\let\lam=\lambda

\let\r=\rho

\let\f=\frac

\let\om=\omega

\let\Om=\Omega

\let\tri=\triangle
\let\ep=\varepsilon
\def\cB{{\mathcal B}}

\def\cH{{\mathcal H}}

\def\cM{{\mathcal M}}

\def\cO{{\mathcal O}}

\def\virgp{\raise 2pt\hbox{,}}
\def\cdotpv{\raise 2pt\hbox{;}}

\def\C{\mathop{\mathbb C\kern 0pt}\nolimits}
\def\DD{\mathop{\mathbb D\kern 0pt}\nolimits}
\def\EE{\mathop{{\mathbb E \kern 0pt}}\nolimits}
\def\K{\mathop{\mathbb K\kern 0pt}\nolimits}
\def\N{\mathop{\mathbb N\kern 0pt}\nolimits}
\def\Q{\mathop{\mathbb Q\kern 0pt}\nolimits}
\def\R{\mathop{\mathbb R\kern 0pt}\nolimits}
\def\SS{\mathop{\mathbb S\kern 0pt}\nolimits}
\def\ZZ{\mathop{\mathbb Z\kern 0pt}\nolimits}
\def\T{\mathbb{T}}
\def\P{\mathop{\mathbb P\kern 0pt}\nolimits}
\def\I{\mathop{\mathbb I\kern 0pt}\nolimits}

\def\dv{\mbox{div}}

\def\dive{\mathop{\rm div}\nolimits}

\def\no{\noindent}
\def\na{\nabla}

\newcommand{\beq}{\begin{equation}}
\newcommand{\eeq}{\end{equation}}
\newcommand{\beo}{\begin{equation*}}
\newcommand{\eeo}{\end{equation*}}
\newcommand{\ben}{\begin{eqnarray}}
\newcommand{\een}{\end{eqnarray}}
\newtheorem{thm}{Theorem}[section]
\newtheorem{lem}{Lemma}[section]
\newtheorem{rmk}{Remark}[section]

\newtheorem{prop}{Proposition}[section]

\begin{document}

\title[Global solution to 2D compressible NS equations]
{ Global classical solution to the Cauchy problem of 2D baratropic compressible Navier-Stokes system with large initial data }
\author[J. Huang]{Jingchi Huang}\address[J. HUANG]
{Department of Mathematics, Pennylvania State University\\
University Park, PA 16802, USA} \email{juh43@psu.edu}
\author[C.Wang]{Chao Wang}
\address [C. WANG]
{Beijing International Center for Mathematical Research\\
Peking University, Beijing 100871, P. R. CHINA}
\email{wangchao@amss.ac.cn}
%\date{28/Jun/2014}
\maketitle
\begin{abstract}
For periodic initial data with initial density, we establish the global existence and uniqueness of strong and classical solutions for the two-dimensional compressible Navier-Stokes equations with no restrictions on the size of initial data provided the shear viscosity is a positive constant and the bulk one is $\lam=\rho^{\b}$ with $\b>1.$
\end{abstract}

\noindent {\sl Keywords:} Compressible  Navier-Stokes equations; global strong solutions; large initial data.

\vskip 0.2cm

%\noindent {\sl AMS Subject Classification (2000):} 35Q30, 76D03  \\

\setcounter{equation}{0}
\section{Introduction}

We study the two-dimensional barotropic compressible Navier-Stokes equations which read as follows \beq\label{CNS} \left\{
\begin{array}
{l} \displaystyle \pa_t\rho + \dive(\rho u) = 0,\\
\displaystyle \pa_t (\rho u) + \dive(\rho u \otimes u) + \nabla P= \mu\tri u+\nabla((\mu+\lam)\dive u),
\end{array}
\right. \eeq where $t\geq 0,$ $x=(x_1,x_2)\in \T^2,$  $\rho,$ $u=(u_1, u_2)$ stand for the density and
velocity of the fluid respectively, and the pressure $P$
is given by
\beq\label{pressure}
P(\rho)=R\rho^{\g},\qquad \g>1.
\eeq
The shear viscosity $\mu$ and the bulk one $\lam$ satisfy the following hypothesis:
\beq\label{viscosity}
0<\mu= const,\quad \lam(\rho)=b\rho^{\b},\quad b>0, \quad \b>0.
\eeq
In the sequel, we set $R=b=1$ without losing any generality.

We consider the Cauchy problem with the given initial data $\rho_0$ and $m_0,$ which are periodic with period $1$ in each space direction $x_i,$ $i=1,2,$ i.e., functions defined on $\T^2=\mathbb{R}^2/\mathbb{Z}^2.$ We require that
\beq\label{initial}
\rho(x,0)=\rho_0(x), \quad \rho u(x,0)=m_0(x),\quad x\in\T^2.
\eeq

The compressible Navier-Stokes equations have a very long history. There is a huge literature concerning the theory of the weak solutions to \eqref{CNS}. Hoff proved the global existence of weak solution for the discontinuous initial data with small energy in  \cite{Hoff-JDE, Hoff-JMFM}. For the large initial data, the global existence of weak solution was proved by Lions \cite{Lions} for the isentropic Navier-Stokes equation, i.e. $P=R\rho^\gamma$ for $\gamma\ge \f95.$ Jiang and Zhang \cite{Jiang1,Jiang2} proved the global existence of weak solution for any $\gamma>1$ for the spherically symmetric or axisymmetric initial data. Inspired by the Jiang and Zhang's work,  Feireisl, Novotn\'{y} and
Petzeltov\'{a} \cite{FNP} improved Lions's result to $\gamma>\f32.$ However, the question of the regularity and uniqueness of weak solutions is completely open even in the case of two dimensional space.

Compared to the weak solutions, the results on the strong solutions are much less. 1962, Nash proved the local existence and uniqueness of smooth solution of the system (\ref{CNS}) for smooth initial data without vacuum in \cite{Nash}.
In a seminal paper \cite{MN}, Matsumura and Nishida proved that the solution is global in time
if the initial data is close to equilibrium. However, whether smooth solutions with large initial data
blow up in finite time is an open problem. Xin \cite{Xin} proved that
smooth solution of the full compressible Navier-Stokes equations will blow up in finite time if the initial density has compact support. Recently, Sun, Wang and Zhang \cite{Sun} showed that smooth solution does not blow up
if the upper bound of the density is bounded, see \cite{Sun-ARMA} for the heat-conductive flow.

For the global existence results in \cite{Charve, CMZ-Rev, CMZ-CPAM, Dan-Inve, MN}, the initial density is required to be
close to a positive constant in $L^\infty$ norm, hence precluding the large oscillation of the density at any
point. Recently, Fang and Zhang \cite{FZ} proved the global existence and uniqueness of (\ref{CNS})
for the initial density $\rho_0$ close to a positive constant in $L^2$ norm and $u_0$ small in $L^p$ norm for $p>3,$
hence allowing the density to have large oscillation on a set of small measure. Similar result
has also been obtained by Huang, Li and Xin \cite{Huang} for the initial data with vacuum,
but a compatible condition is imposed on the initial data. Recently, Wang, Wang and Zhang \cite{WWZ} prove the global well-posedness for some classes of large initial data.

For the global well-posedness of the strong solutions with general large initial data, Vaigant and Kazhikhov \cite{VK} proved existence of global solutions for system \eqref{CNS}--\eqref{viscosity} with $\beta>3$. Recently, Huang and Li \cite{HL, HL1} relax the restriction $\beta>3$ to $\beta>\f43.$ In this paper, we aim to relax this restriction to $\beta>1.$

Before stating the main results, we explain the notations and conventions used throughout this paper. We denote
\begin{equation*}
\int f\,dx=\int_{\T^2}\,dx,\quad \bar{f}=\f1{|\T^2|}\int f\,dx.
\end{equation*}
For $1\leq r\leq \infty,$ we also denote the standard Lebesgue and Sobolev spaces as follows:
\begin{equation*}
L^r=L^r(\T^2),\quad W^{s,r}=W^{s,r}(\T^2), \quad H^s=W^{s,2}.
\end{equation*}

\begin{thm}\label{mainthm1} {\sl 
Assume that
\beq\label{thmassume1}
\b>1,\qquad \g>1,
\eeq
and that the initial data $(\rho, u_0)$ satisfies that for some $q>2$,
\beno
0<\rho_0\in W^{1, q}, \quad u_0\in H^2. 
\eeno
Then, the system \eqref{CNS}--\eqref{initial} has a unique global strong solution $(\rho, u )$ satisfying that
\beo
\begin{cases}
\rho\in C([0,T]; W^{1, q}), \quad \rho_t\in C([0,T]; L^2),\\
u\in L^2(0,T; H^3), \quad u\in C([0,T]; H^2),\\
\end{cases}
\eeo
for any $0<T<\infty$.
}
\end{thm}

\no{\bf The organization of the paper.} In Section 2, we collect some elementary facts and inequalities which will be needed in later analysis. Section 3 is devoted to derivation of upper bound on the density which is the key to extend the local solution to all time. Based on the previous estimates, higher-order ones are established in Section 4. We also prove the blow-up criterion and the main result, Theorem \ref{mainthm1}, in this Section.

\setcounter{equation}{0}
\section{Preliminaries}
The following well-known local existence theory, where the initial density is strictly away from vacuum, can be found in \cite{SS, Solon}.

\begin{lem}\label{locallem}
{\sl Assume that $(\rho_0, u_0)$ satisfies
\beq\label{localassume1}
\rho_0\in W^{1, q},\quad u_0\in H^2,\quad \inf_{x\in \T^2}\rho_0(x)>0,\quad m_0=\rho_0 u_0.
\eeq
Then there are a small time $T>0$ and a constant $C_0>0$ both depending only on $\|\rho_0\|_{H^2},$ $\|u_0\|_{H^2},$ and $\inf_{x\in \T^2}\rho_0(x)$ such that there exists a unique strong solution $(\rho, u)$ to the problem \eqref{CNS}-\eqref{initial} in $\mathbb{T}^2 \times (0,T)$ satisfying
\beq\label{localconclu1}
\begin{cases}
\rho\in C([0,T]; W^{1, q}), \quad \rho_t\in C([0,T]; L^2),\\
u\in L^2(0,T; H^3), \quad u\in C([0,T]; H^2),
\end{cases}
\eeq
and
\beq\label{localconclu2}
\inf_{(x,t)\in \T^2\times (0,T)}\rho(x,t)>C_0>0.
\eeq}
\end{lem}

\begin{rmk}\label{localrmk1}
{\sl It should be mentioned that \cite{SS, Solon} dealt with the case that $\lam=$ const. However, after some slight modifications, their methods can be also be applied to the problem \eqref{CNS}-\eqref{initial}.}
\end{rmk}

%\begin{rmk}\label{localrmk2}
%{\sl In \cite{SS,Solon}, instead of $\eqref{localconclu1}_1,$ it was shown that
%\begin{equation*}
%\rho\in L^\infty(0,T; H^2), \quad \rho_t\in L^\infty(0,T; H^1).
%\end{equation*}
%However, one can use \cite{Lions1} Lemma 2.3 to derive $\eqref{localconclu1}_1$ by standard arguments (see \cite{CCK} for details). Moreover, one can also obtain $\eqref{localconclu1}_3$ by standard arguments due to $\eqref{localconclu1}_1,$ $\eqref{localconclu1}_2,$ and $\eqref{localconclu2}.$}
%\end{rmk}

The following Poincare-Sobolev and Brezis-Wainger inequalities will be used frequently.
\begin{lem}\label{interpolationlem1}[\cite{BW, Eng, LSU}]
{\sl There exists a positive constant $C$ depending only on $\T^2$ such that every function $u\in H^1(\T^2)$ satisfies for $2<p<\infty,$
\beq\label{psineq}
\|u-\bar{u}\|_{L^p}\leq Cp^{\f12}\|u-\bar{u}\|_{L^2}^{\f2p}\|\na u\|_{L^2}^{1-\f2p}, \quad \|u\|_{L^p}\leq Cp^{\f12}\|u\|_{L^2}^{\f2p}\|u\|_{H^1}^{1-\f2p}.
\eeq
Moreover, for $q>2,$ there exists some positive constant $C$ depending only on $q$ and $\T^2$ such that every function $v\in W^{1,q}(\T^2)$ satisfies
\beq\label{bwineq}
\|v\|_{L^\infty}\leq C\|v\|_{H^1}\log^{\f12}(e+\|\na v\|_{L^q})+C.
\eeq}
\end{lem}

The following Poincare type inequality can be found in \cite{Fei}.
\begin{lem}\label{interpolationlem2}
{\sl Let $u\in H^1(\T^2),$ and $\rho$ be a non-negative function such that
\begin{equation*}
0<M_1\leq \int\rho\,dx, \quad \int\rho^{\g}\,dx\leq M_2,
\end{equation*}
with $\g>1.$ There exists a positive constant $C$ depending only on $M_1$ and $M_2$ such that
\beq\label{feiineq}
\|u\|^2_{L^2}\leq C\int\rho u^2\,dx+C\|\na u\|^2_{L^2}.
\eeq
}
\end{lem}

Then, we state the following Beale-Kato-Majda type inequality which was proved in \cite{BKM} when $\dive u\equiv 0$ and will be used later to estimate $\|\na u\|_{L^\infty}$ and $\|\na \rho\|_{L^p}.$
\begin{lem}\label{bkmlem}
{\sl For $2<p<\infty,$ there is a constant $C(p)$ such that  the following estimate holds for all $\na u\in W^{1,p}(\T^2),$
\beq\label{bkmineq}
\|\na u\|_{L^\infty}\leq C(\|\dive u\|_{L^\infty}+\|\mathrm{rot} u\|_{L^\infty})\log (e+\|\na^2 u\|_{L^p})+C\|\na u\|_{L^2}+C.
\eeq
}
\end{lem}

Next, let $\tri^{-1}$ denote the Laplacian inverse with zero mean on $\T^2$ and $R_i$ be the usual Riesz transform on $\T^2:$ $R_i=(-\tri)^{-\f12}\pa_i.$ Let $\cH^1(\T^2)$ and $\cB\cM\cO(\T^2)$ stand for the usual Hardy and BMO space:
\begin{equation*}\begin{split}
&\cH^1=\{f\in L^1(\T^2): \|f\|_{\cH^1}=\|f\|_{L^1}+\|R_1f\|_{L^1}+\|R_2 f\|_{L^1}<\infty, \bar{f}=0         \},\\
&\cB\cM\cO=\{f\in L^1_{loc}(\T^2): \|f\|_{\cB\cM\cO}<\infty\},
\end{split}
\end{equation*}
with
\begin{equation*}
\|f\|_{\cB\cM\cO}=\sup_{x\in \T^2, r\in (0,d)} \f1{|\Omega_r(x)|}\int_{\Omega_r(x)}\Big| f(y)-\f1{|\Omega_r(x)|}\int_{\Omega_r(x)} f(z)\,dz\Big| \,dy,
\end{equation*}
where $d$ is the diameter of $\T^2,$ $\Omega_r(x)=\T^2\cap B_r(x),$ and $B_r(x)$ is a ball with center $x$ and radius $r.$ Given a function $b,$ define the linear operator
\begin{equation*}
[b, R_iR_j](f)\triangleq bR_i\circ R_j(f)-R_i\circ R_j(bf),\quad i,j=1,2.
\end{equation*}
The following properties of the commutator $[b, R_iR_j](f),$ which are due to \cite{CM,CRW} respectively, will be useful for our work.
\begin{lem}\label{commutelem}
{\sl Let $b,f\in C^{\infty}(\T^2).$ Then for $p\in (1,\infty),$ there is $C(p)$ such that
\beq\label{commuteineq1}
\|[b, R_iR_j](f)\|_{L^p}\leq C(p)\|b\|_{\cB\cM\cO}\|f\|_{L^p}.
\eeq
Moreover, for $q_i\in (1,\infty)$ ($i=1,2,3$) with $q_1^{-1}=q_2^{-1}+q_3^{-1},$ there is a $C$ depending only on $q_i$ such that
\beq\label{commuteineq2}
\|\na [b, R_iR_j](f)\|_{L^{q_1}}\leq C\|\na b\|_{L^{q_2}}\|f\|_{L^{q_3}}.
\eeq
}
\end{lem}

\setcounter{equation}{0}
\section{A Priori Estimates: Upper Bound of the Density}
First, we have the following standard energy inequality.
\begin{lem}\label{energylem1}
{\sl There exists a positve constant $C$ depending only on $\g,$ $T$ $\|\rho_0\|_{L^\g},$ and $\|\rho_0^{\f12} u_0\|_{L^2}$ such that
\beq\label{energyineq}
\sup_{0\leq t\leq T}\int (\rho |u|^2+\rho^\g)\,dx +\int_0^T\int \bigl(\mu|\na u|^2+\lam(\rho)(\dive u)^2\bigr)\,dx\,dt\leq C.
\eeq}
\end{lem}
Next, we state the $L^p$ esitmate of the density due to Vaigant-Kazhikhov \cite{VK}.
\begin{lem}\label{vklem}
{\sl Let $\b>1,$ for any $1<p<\infty,$ there is a positive constant $C(T)$ depending only on $T,\mu,\b,\g,$ and
\begin{equation*}
E_0\triangleq \|\rho_0\|_{L^\infty}+\|\rho_0^{\f12} u_0\|_{L^2}+\|\na u_0\|_{L^2}
\end{equation*}
such that
\beq\label{vkineq}
\sup_{0\leq t\leq T}\|\rho(\cdot, t)\|_{L^p}\leq C(T)p^{\f2{\b-1}}.
\eeq}
\end{lem}

To proceed, we denote by
\begin{equation*}
\na^{\perp} =(\pa_2, -\pa_1),\qquad \f{D}{Dt}f=\dot{f}=f_t+u\cdot\na f,
\end{equation*}
where $\f{D}{Dt}f$ is the material derivative of $f.$ Let $G$ and $\omega$ be the effective viscous flux and the vorticity respectively as follows:
\begin{equation*}
G\triangleq (2\mu +\lam(\rho))\dive u-(P-\bar{P}),\quad \omega\triangleq \na^{\perp}\cdot u=\pa_2u_1-\pa_1u_2.
\end{equation*}
Then we rewirte the momentum equation $\eqref{CNS}_2$ as
\beq\label{ueq}
\rho\dot{u}=\na G+\mu\na^{\perp}\omega,
\eeq
which shows that $G$ solves
\beo
\tri G=\dive(\rho \dot{u})=\pa_t(\dive(\rho u))+\dive\dive(\rho u\otimes u).
\eeo
This implies
\beq\label{rhoeq}
G-\bar{G}+\f{D}{Dt}\bigl((-\tri)^{-1}\dive(\rho u)\bigr)=F,
\eeq
where $F$ is a commutator defined by
\beq\label{defiF}
F\triangleq\sum_{i,j=1}^2[u_i, R_iR_j](\rho u_j)=\sum_{i,j=1}^2 u_iR_i\circ R_j(\rho u_j)-R_i\circ R_j(\rho u_iu_j).
\eeq
The mass equation $\eqref{CNS}_1$ leads to
\beo
-\dive u=\f1{\rho}\f{D}{Dt}\rho,
\eeo
which combining with \eqref{rhoeq} gives that
\beq\label{rhoeq2}
\f{D}{Dt}\varphi(\rho)+P=\f{D}{Dt}\psi +\bar{P}-\bar{G}+F,
\eeq
with
\beq\label{defith}
\varphi(\r)\triangleq 2\mu\log\r+\b^{-1}\r^{\b},\quad \psi\triangleq(-\tri)^{-1}\dive(\r u).
\eeq

\begin{lem}\label{energylem2}
{\sl Assume that \eqref{thmassume1} holds. Then there is a constant $C$ depending only on $\mu,\b,\g, T,$ and $E_0$ such that
\beq\label{energyest1}
\sup_{0\leq t\leq T}\log (e+A^2(t))+\int_0^T\f{B^2(t)}{e+A^2(t)}\,dt\leq C R_T^{1+\delta\b}, \quad \delta\in (0,1),
\eeq
where
\beq\label{defiAB}
A^2(t)\triangleq\int\Bigl(\omega^2(t)+\f{G^2(t)}{2\mu+\lam(\rho(t))}\Bigr)\,dx, \quad B^2(t)\triangleq\int\rho(t)|\dot{u}(t)|^2\,dx,
\eeq
and
\beq\label{defiRT}
R_T\triangleq 1+\sup_{0\leq t\leq T}\|\rho(\cdot,t)\|_{L^\infty}.
\eeq
}
\end{lem}
\begin{proof}
First, direct calculations show that
\beq\label{ueq1}
\na^{\perp}\cdot \dot{u}=\f{D}{Dt}\omega-(\pa_1u\cdot\na)u_2+(\pa_2 u\cdot\na)u_1=\f{D}{Dt}\omega+\omega\dive u,
\eeq
and that
\beq\label{ueq2}
\begin{split}
\dive\dot{u}&=\f{D}{Dt}\dive u+(\pa_1u\cdot\na)u_1+(\pa_2u\cdot\na)u_2\\
&=\f{D}{Dt}(\f{G}{2\mu+\lam})+\f{D}{Dt}(\f{P-\bar{P}}{2\mu+\lam})-2\na u_1\cdot\na^{\perp}u_2+(\dive u)^2.
\end{split}
\eeq

Multiplying \eqref{ueq} by $2\dot{u}$ and integrating the resulting equality over $\T^2,$ we obtain after using \eqref{ueq1} and \eqref{ueq2} that
\beq\label{ueq4}
\begin{split}
\f{d}{dt}&A^2+2B^2\\
=&-\int \omega^2\dive u\,dx+4\int G\na u_1\cdot\na^{\perp}u_2\,dx-2\int G(\dive u)^2\,dx\\
&-\int \f{(\b-1)\lam-2\mu}{(2\mu+\lam)^2}G^2\dive u\,dx+2\b\int \f{\lam(P-\bar{P})}{(2\mu+\lam)^2}G\dive u\,dx\\
&-2\g\int\f{P}{2\mu+\lam}G\dive u\,dx +2(\g-1)\int P\dive u\,dx\int \f{G}{2\mu+\lam}\,dx\\
=&\sum_{i=1}^7 I_i.
\end{split}
\eeq

Now, we estimate each $I_i$ as follows:

First, we deduce from \eqref{ueq} to get that
\begin{equation*}
\tri G=\dive(\rho \dot{u}),\quad \mu\tri \omega=\na^{\perp}\cdot(\rho\dot{u}),
\end{equation*}
and use with standard $L^p$ estimate of elliptic equations to obtain that for any $p\in (1,\infty),$
\beq\label{ellipticest1}
\|\na G\|_{L^p}+\|\na \omega\|_{L^p}\leq C(p,\mu)\|\rho\dot u\|_{L^p}.
\eeq
In particular, we have
\beq\label{ellipticest2}
\|\na G\|_{L^2}+\|\na \omega\|_{L^2}\leq C(\mu)\|\rho\dot u\|_{L^2}\leq C(\mu)R_T^{\f12}B.
\eeq
This combining with \eqref{psineq} gives
\beq\label{omegaest}
\|\omega\|_{L^4}\leq C\|\omega\|_{L^2}^{\f12}\|\na \omega\|_{L^2}^{\f12}\leq C R_T^{\f14}A^{\f12}B^{\f12},
\eeq
which leads to
\beq\label{I1est}
|I_1|\leq C\|\omega\|_{L^4}\|\dive u\|_{L^2}\leq \ep B^2+C(\ep)R_T\|\na u\|_{L^2}^2 A^2.
\eeq

Next, we will use an idea due to \cite{Desj, Pere} to estimate $I_2.$ Noticing that
\begin{equation*}
\mathrm{rot}\na u_1=0,\qquad \dive \na^{\perp}u_2=0,
\end{equation*}
one can derives from \cite{CLMS} that
\beo
\|\na u_1\cdot\na^{\perp}u_2\|_{\cH^1}\leq C\|\na u\|_{L^2}^2.
\eeo
Together with the fact that $\cB\cM\cO$ is the dual space of $\cH^1$ (see \cite{Feff}), we obtain
\beq\label{I2est}
\begin{split}
|I_2|&\leq C\|G\|_{\cB\cM\cO}\|\na u_1\cdot\na^{\perp}u_2\|_{\cH^1}\leq C\|\na G\|_{L^2}\|\na u\|_{L^2}^2\\
&\leq CR_T^{\f12} B\|\na u\|_{L^2}(1+A)\leq \ep B^2+C(\ep)R_T\|\na u\|_{L^2}^2(1+A^2),
\end{split}
\eeq
where in the third inequality we have use \eqref{ellipticest2} and the following simple fact that for $t\in [0,T],$
\beq\label{Aest}
C^{-1}\|\na u(\cdot, t)\|_{L^2}^2-C\leq A^2(t)\leq CR_T^{\b}\|\na u(\cdot, t)\|_{L^2}^2+C
\eeq
due to \eqref{vkineq}.

Next, H\"older's inequality, \eqref{pressure} and \eqref{vkineq} yield that for $\delta\in (0, 1),$
\beq\label{Iest}
\begin{split}
\sum_{i=3}^7|I_i|&\leq C\int |\dive u|\Bigl(|G|\f{|G+P-\bar{P}|}{2\mu+\lam}+\f{G^2}{2\mu+\lam}+\f{P|G|}{2\mu+\lam}\Bigr)\,dx+C\int P|\dive u|\,dx\int \f{|G|}{2\mu+\lam}\,dx\\
&\leq C\|\na u\|_{L^2}\|\f{G^2}{2\mu+\lam}\|_{L^2}+C\|\na u\|_{L^2}\|P\|_{L^{2+\delta}}\|G\|_{L^{\f{2(2+\delta)}{\delta}}}+C\|\na u\|_{L^2}\|P\|_{L^2}\|G\|_{L^{\f{2(2+\delta)}{\delta}}}\\
&\leq  C\|\na u\|_{L^2}\|\f{G^2}{2\mu+\lam}\|_{L^2}+C\|\na u\|_{L^2}\|G\|_{L^{\f{2(2+\delta)}{\delta}}}.
\end{split}
\eeq

Then, noticing that \eqref{defiAB} gives
\beq\label{gest1}
\|G\|_{L^2}\leq CR_T^{\f{\b}2}A,
\eeq
which together with the H\"older inequality, \eqref{psineq} and \eqref{ellipticest2} yield that for $0<\delta<1,$
\beq\label{gest2}
\begin{split}
\|\f{G^2}{2\mu+\lam}\|_{L^2}&\leq C\| \Bigl(\f{G}{\sqrt{2\mu+\lam}}\Bigr)^{1-\delta}  G^{1+\delta} \|_{L^2}\leq C\|\f{G}{\sqrt{2\mu+\lam}}\|_{L^2}^{1-\delta}\|G\|^{1+\delta}_{L^{\f{2(1+\delta)}{\delta}}}\\
&\leq CA^{1-\delta}\|G\|_{L^2}^{\delta}\|\na G\|_{L^2}\leq C R_T^{\f{1+\delta\b}2}AB.
\end{split}
\eeq
Similarly, we have
\beq\label{gest3}
\|G\|_{L^{\f{2(2+\delta)}{\delta}}}\leq C\|G\|_{L^2}^{\f{\delta}{2+\delta}}\|\na G\|_{L^2}^{\f{2}{2+\delta}}\leq CR_T^{\f{1+\delta\b}2}A^{\f{\delta}{2+\delta}}B^{\f{2}{2+\delta}}.
\eeq
Putting \eqref{gest2}, and \eqref{gest3} into \eqref{Iest} yields
\beq\label{Iest2}
\begin{split}
\sum_{i=3}^7|I_i|&\leq CR_T^{\f{1+\delta\b}2}\|\na u\|_{L^2} (AB+A^{\f{\delta}{2+\delta}}B^{\f{2}{2+\delta}})\\
&\leq CR_T^{\f{1+\delta\b}2}\|\na u\|_{L^2} (AB+A+B)\\
&\leq \ep B^2 +C(\ep)R_T^{1+\delta \b}(1+\|\na u\|_{L^2}^2)(1+A^2).
\end{split}
\eeq

Finally, substituting \eqref{I1est}, \eqref{I2est}, and \eqref{Iest2} into \eqref{ueq4}, choosing $\ep$ sufficienlty small and $\delta\in (0,1),$ we obtain
\ben\label{ABest}
\f{d}{dt}A^2+ B^2\leq CR_T^{1+\delta\b}(1+\|\na u\|_{L^2}^2)(1+A^2).
\een
Dividing this inequality by $e+A^2,$ and using \eqref{energyineq}, we reach \eqref{energyest1} and finish the proof of Lemma \ref{energylem2}.
\end{proof}

The following $L^p$ estimates of the momentum will play an important role in the estimate of the upper bound of the density.
\begin{lem}\label{momentlem}
{\sl For any $p>2,$ there exists a positive constant $C$ depending only on $p,\mu,\b,\g, T,$ and $E_0$ such that
\beq\label{momentest}
\|\r u\|_{L^p}\leq CR_T(e+\|\na u\|_{L^2}).
\eeq
Moreover, let $\al=\f{\mu^{\f12}}{2(\mu+1)}R_T^{-\f{\b}2}\in (0,\f14],$ then for any $q>3$ and $\ep>0$ there exists a positive constant $C_1$ depending only on $q,\mu,\b,\g, T,$ and $E_0$ such that
\beq\label{momentest1}
\|\r u\|_{L^q}\leq C_1 R_T^{1-\f1q+\ep}(e+\|\na u\|_{L^2})^{1-\f{2+\al}{q}+\ep}\log^{\f{2+\al+q}{2q}}\Bigl(e+\bigl( \f{B^2}{e+A^2}\bigr)^{\f14}\Bigr),
\eeq
where $A,B$ are defined as \eqref{defiAB}.}
\end{lem}
\begin{proof}
\eqref{momentest} is the directly consequence from \eqref{psineq} and \eqref{feiineq}. We just focus on the proof of \eqref{momentest1}.

Multiplying $\eqref{CNS}_2$ by $(2+\al)|u|^{\al}u,$ we get after integrating the resulting equation over $\T^2$ that
\beo
\begin{split}
\f{d}{dt}&\int \r |u|^{2+\al}\,dx+(2+\al)\int |u|^{\al}\bigl(\mu |\na u|^2+(\mu +\lam)(\dive u)^2\bigr)\,dx\\
\leq& (2+\al)\al\int (\mu+\lam)|\dive u||u|^{\al}|\na u|\,dx+ C\int \r^{\g}|u|^{\al}|\na u|\,dx\\
\leq& \f{2+\al}2 \int  (\mu+\lam)(\dive u)^2|u|^{\al}\,dx+(\f{(2+\al)\mu}{8(\mu+1)}+\mu) \int |u|^{\al}|\na u|^2\,dx\\
&+C\int  \r |u|^{2+\al}\,dx+C\int \r^{(2+\al)\g-\f{\al}2}\,dx,
\end{split}
\eeo
which together with Gronwall's inequality and \eqref{vkineq} thus gives
\beq\label{energyest2}
\sup_{0\leq t\leq T}\int \r |u|^{2+\al}\,dx\leq C.
\eeq
Then let $s=1-\f{2+\al}{q},$ it follows form H\"older's inequality that
\beq\label{momentest2}
\begin{split}
\|\r u\|_{L^q}&\leq C\|\r u\|_{L^{2+\al}}^{1-s}\|\r u\|_{L^\infty}^{s}\\
&\leq C R_T^{\f{1+\al}{2+\al}(1-s)} R_T^{s} \|u\|_{L^\infty}^s =CR_T^{1-\f1q}\|u\|_{L^\infty}^s.
\end{split}\eeq

Using \eqref{bwineq}, \eqref{feiineq}, and \eqref{energyineq}, we have
\beq\label{momentest3}
\|u\|_{L^\infty} \leq C\|u\|_{H^1}\log^{\f12}(e+\|\na u\|_{L^4}) \leq C(1+\|\na u\|_{L^2})\log^{\f12}(e+\|\na u\|_{L^4}).
\eeq
But from \eqref{omegaest}, \eqref{Aest}, \eqref{gest2}, and \eqref{vkineq}, we obtain that
\beq\label{uest}
\begin{split}
\|\na u\|_{L^4}&\leq C(\|\dive u\|_{L^4}+\|\omega\|_{L^4}) \leq C\|\f{G+P-\bar{P}}{2\mu+\lam}\|_{L^4}+CR_T^{\f14}A^{\f12}B^{\f12}\\
&\leq C\|\f{G^2}{2\mu+\lam}\|_{L^2} +C + CR_T^{\f14}A^{\f12}B^{\f12} \leq CR_T^{\f{1+\delta\b}4}A^{\f12}B^{\f12}\\
&\leq CR_T^{\f{1+\delta\b}4} (e+A^2)^{\f12} \Bigl( \f{B^2}{e+A^2}\Bigr)^{\f14}\leq CR_T^{\f{1+2\b+\delta\b}4} (e+\|\na u\|_{L^2}) \Bigl( \f{B^2}{e+A^2}\Bigr)^{\f14}.
\end{split}\eeq
Substituting \eqref{momentest3} and \eqref{uest} into \eqref{momentest2}, we obtain \eqref{momentest1}, which completes the proof of Lemma \ref{momentlem}.
\end{proof}

Now we are in the position to prove the main result of this section.
\begin{prop}\label{densityprop}
{\sl Under the conditions of Theorem \ref{mainthm1}, there is a constant $C$ depending only on $\mu, \b, \g, T,$ and $E_0$ such that
\beq\label{densityest}
\sup_{0\leq t\leq T}(\|\r(t)\|_{L^\infty}+\|\na u(t)\|_{L^2})+\int_0^T\int \r |\dot{u}|^2\,dx\,dt\leq C.
\eeq}
\end{prop}
\begin{proof}
First, it follows from \eqref{energyineq} and \eqref{vkineq} that
\beq\label{energyest3}
\|\r u\|_{L^{\f{2\al}{\al+1}}}\leq C\|\r\|_{L^{\al}}^{\f12}\|\r^{\f12}u\|_{L^2}\leq C,
\eeq
which together with \eqref{bwineq}, \eqref{energyest1}, \eqref{Aest}, and \eqref{momentest} yields that for any $\ep>0,$
\beq\label{psiest}
\begin{split}
\|\psi\|_{L^\infty}&\leq C\|\psi\|_{H^1}\log^{\f12}(e+\|\na \psi\|_{L^3}) +C\\
&\leq C(\|\r u\|_{L^{\f{2\al}{\al+1}}}+\|\r u\|_{L^2})\log^{\f12}(e+\|\r u\|_{L^3})+C\\
&\leq CR_T^{\f12}\log^{\f12}\bigl(R_T(e+\|\na u\|_{L^2})\bigr)+C\\
&\leq CR_T^{\f12+\ep}\log^{\f12}(e+A^2)+C\\
&\leq CR_T^{1+\f{\delta\b}2+\ep}.
\end{split}
\eeq
Next, on one hand, we deal with the $\|F\|_{L^\infty}$. Taking $\f{1}{p}+\f{1}{q}=\f12$ and Applying the \eqref{bwineq}, we have that
\beq\label{F}
\begin{split}
\|F\|_{L^\infty}&\leq C\|F\|_{H^1}\log^{\f12}(e+\|\na F\|_{L^3}) +C\\
&\leq C(\|\na u\|_{L^2}\|\rho u\|_{L^2}+\|\na u\|_{L^q}\|\rho u\|_{L^p})\log^{\f12}(e+\|\na u\|_{L^4}\|\rho u\|_{L^{12}})+C\\
&\leq C(R_T^{\f12}\|\na u\|_{L^2} \|\rho^{\f12} u\|_{L^2}  +\|\na u\|_{L^q}\|\rho u\|_{L^p})\log^{\f12}(e+\|\na u\|_{L^4}\|\rho u\|_{L^{12}})+C.
\end{split}
\eeq
By the H\"older inequality and \eqref{uest}, we have that
\beno
\|\na u\|_{L^q}
&\leq& \|\nabla u\|^\theta_{L^2}\|\nabla u\|^{1-\theta}_{L^4}\leq C \|\nabla u\|^\theta_{L^2}(R_T^{\f{1+\delta\b}4+\f \beta 2} \|\nabla u\|_{L^2} (\f{B}{e+A})^{\f12}+1)^{1-\theta}\\
&\leq& C(1+ \|\nabla u\|_{L^2})(R_T^{\f{1+\delta\b}4+\f \beta 2}  (\f{B}{e+A})^{\f12}+1)^{1-\theta},
\eeno
where $\f{1}{q }=\f{\theta}{2}+\f{1-\theta}{4}$.

Next, by \eqref{momentest1}, we have that  for $s=1-\f{2+\al}{p},$
\beno
\|\rho u\|_{L^p}&\leq&  C R_T^{1-\f1p+\e}\|u\|^{s+\e}_{H^1}\left(\log(1+    (\f{B}{e+A})^{\f12} )\right)^{s}.
\eeno

Combining above two estimates, we have that
\beno
\|\na u\|_{L^q}\|\rho u\|_{L^p}
&\leq&  CR_T^{1-\f1p+(1-\theta)(\f{1+\delta\b}4+\f \beta 2)+\e} (1+ \|u\|_{H^1})^{1+s+\e}( 1+ (\f{B}{e+A}) )^{\f{1-\theta+\e}{2}}  .
\eeno

By the definition of $s, \theta$, we have that
\beno
s=1-(2+\al)\f{1-\theta}{4},
\eeno
which implies that when $0<\theta<1$
\beno
2s< 1+\theta.
\eeno
Taking $\e$ small enough such that
\beno
2s+3\e\leq 1+\theta.
\eeno
Thus, we have
\beno
\int^T_0\|F\|_{L^\infty}\,dt \leq C\left(R_T^{1-\f1p+(1-\theta)(\f{1+\delta\b}4+\f \beta 2)+\e}  \right)^{\f {4}{\theta+3-\e}}\int^T_0 \|u\|_{H^1}^{2}\,dt+C\int_0^T(1+ \f{B}{e+A})^2\,dt.
\eeno
If we take $\theta$ close to 1, we have that
\beno
\int^T_0\|F\|_{L^\infty}\,dt +\|\psi\|_{L^\infty} \leq  C(R_T^{1+\f{\delta\b}2+\ep}+1).
\eeno

Now, we are in the position to prove the upper bound of the density. Recalling the \eqref{rhoeq2} and integrating to get that
\beno
R_T^{\beta}
& \leq & C(R_T^{1+\f{\delta\b}2+\ep}+1)+C\int_0^T(1+|\overline{G}|)\,dt\\
&\leq &C(R_T^{1+\f{\delta\b}2+\ep}+1)+\int_0^T(1+\|\rho^\beta\|_{L^2}\|\dv u\|_{L^2})\,dt\\
&\leq &C(R_T^{1+\f{\delta\b}2+\ep}+1).
\eeno
where the constant depends on the initial data and the time $T$. Thus, if we take $\beta>1,$ $\delta$ and $\ep$ sufficiently small such that $\b>1+\f{\delta\b}2+\ep,$ we can get that
\beno
R_T\leq C.
\eeno

By the inequality \eqref{ABest}, we have that
\beno
\sup_{0\leq t\leq T} \big(e+A^2(t)\big)+\int_0^T {B^2(t)} \,dt\leq C .
\eeno

By the classic elliptic estimate, we get the desired result.
\end{proof}

\setcounter{equation}{0}
\section{A Blow-up Criterion}

In this section, we will establish a blow-up criterion for the \eqref{CNS}--\eqref{viscosity} for all $\beta>1.$
\begin{prop}\label{Blow-upC}
{\sl Assume that $(\rho,u)$ is the strong solution of \eqref{CNS}--\eqref{viscosity}. Let $T^*$ be a maximal existence time of the solution.
If $T^*<\infty$, then we have
\begin{eqnarray}\label{blowup}
\limsup\limits_{T\uparrow
T^*}\|\rho(x,t)\|_{L^{\infty}(0,T;L^\infty(\Omega))}=\infty.
\end{eqnarray}}
\end{prop}

Fist, from the proof of the Lemma \ref{energylem1} and Lemma \ref{energylem2}, we obtain
\begin{prop}
{\sl There exists a constant $C$ depending only on $\mu, \b, \g, T,$ and $E_0, \|\rho\|_{L^\infty}$ such that
\beno
&&\sup_{0\leq t\leq T}\int (\rho |u|^2+\rho^\g)(t)\,dx +\int_0^T\int \bigl(\mu|\na u|^2+\lam(\rho)(\dive u)^2\bigr)\,dx\,dt\leq C,\\
&&\sup_{0\leq t\leq T}\big(e+A^2(t)\big)+\int_0^TB^2(t)\,dt\leq C.
\eeno}
\end{prop}

Next, we have the second order estimates as following:
\begin{prop}\label{highest1}
{\sl There exists a constant $C$ depending only on $\mu, \b, \g, T, E_0$ and $\|\rho\|_{L^\infty}$ such that
\beq\label{densityest1}
\sup_{0\leq t\leq T} \|\r^{\f12} \dot{u}(t)\|_{L^2}  +\int_0^T\int  |\nabla \dot{u}|^2\,dx\,dt\leq C.
\eeq}
\end{prop}
\begin{proof}Take the material derivative $\f{D}{Dt}$ on the both side of $\eqref{CNS}_2$ to get,
\beq\label{higheq}
\begin{split}
&(\rho \dot{u}_j)_t+\dv(\rho u \dot{u}_j)-\mu\tri\dot{u}_j-\pa_j((\mu+\lambda)\dv \dot{u}) \\
=&\mu\pa_i(-\pa_i u\cdot \nabla u_j+\dv u \pa_i u_j)-\mu \dv(\pa_i u\pa_i u_j)-\pa_j[(\mu+\lambda)\pa_i u\cdot \nabla u_i-(\mu+(1-\beta)\rho^\beta)(\dv u)^2] \\
&-\dv(\pa_j(\mu+\lambda)\dv u)+(\gamma-1)\pa_j(P\dv u)+\dv (P\pa_j u).
\end{split}
\eeq

Then, multiplying $\dot{u}$ on the both sides of \eqref{higheq}, we get that
\beno
&&\f{d}{dt}\int \rho |\dot{u}|^2\,dx+\mu \int|\nabla \dot{u}|^2\,dx+\int(\mu+\lambda)(\dv \dot{u})^2\,dx\\
&\leq& \e \int|\nabla \dot{u}|^2\,dx+C_\e (\|\nabla u\|_{L^4}^4+ \|\nabla u\|_{L^2}^2).
\eeno
Pluging \eqref{uest} into the above inequality, we obtain
\beno
&&\f{d}{dt}\int \rho |\dot{u}|^2\,dx+\mu \int|\nabla \dot{u}|^2\,dx+\int(\mu+\lambda)(\dv \dot{u})^2\,dx \leq  C_\e (A^2 B^2 + \|\nabla u\|_{L^2}^2).
\eeno
Thus, by Proposition \ref{highest1}, the proof is completed.
\end{proof}

\medskip

Next, we compute the higher order estimates for density:
\begin{prop}\label{highest2}
{\sl There exists a constant $C$ depending only on $\mu, \b, \g, T,$ and $E_0, \|\rho\|_{L^\infty}$ such that
\beq\label{densityest}
\sup_{0\leq t\leq T} (\|\rho\|_{ W^{1, q}}+\|\nabla u\|^2_{H^1} ) +\int_0^T  \|\nabla^2 u\|^2_{L^q} \,dt\leq C.
\eeq
where $q\geq 2$. }
\end{prop}
\begin{proof}
The proof comes from \cite{HL, HL1}. First, we denote $\Phi \triangleq (2\mu+\lambda(\rho))\nabla \rho$, then from the density equation, we have 
\beno
\Phi_t+(u\cdot \nabla)\Phi+(2\mu+\lambda(\rho))\nabla u\cdot \nabla\rho+\rho\nabla(G+P)+\Phi\dv u=0.
\eeno
Multiplying $|\Phi|^{q-2}\Phi$ on the both sides of the above equation and integrating by parts, we obtain
\beq\label{Phi}\begin{split}
\f{d}{dt}\|\Phi\|_{L^q}
&\leq C(1+\|\nabla u\|_{L^\infty})\|\nabla \rho\|_{L^q}+C\|\nabla G\|_{L^q}\\
&\leq C(1+\|\nabla u\|_{L^\infty})\|\nabla \rho\|_{L^q}+C\|\rho\dot{u}\|_{L^q}.
\end{split}\eeq

On one hand, recalling \eqref{ellipticest1}, for any $q>2$ we get
\beo\begin{split}
\|\nabla u\|_{L^\infty}
&\leq C(\|\dv u\|_{L^\infty}+\|\om\|_{L^\infty})\log (1+\|\na^2 u\|_{L^q})+C\|\nabla u\|_{L^2}+C\\
&\leq C(1+\|G\|_{L^\infty}+\|\om\|_{L^\infty})\log (1+\|\na \dv u\|_{L^q}+\|\na \om\|_{L^q})+C\|\nabla u\|_{L^2}+C\\
&\leq C(1+\|\na G\|^{\f{q}{2(q-1)}}_{L^q}+\|\nabla \om\|^{\f{q}{2(q-1)}}_{L^q})\log (1+\|\na(\f{G+P-\overline{P}}{2\mu+\lambda})\|_{L^q}+\|\rho\dot{u}\|_{L^q})+C\|\nabla u\|_{L^2}+C\\
&\leq C(1+\|\rho\dot{u}\|^{\f{q}{2(q-1)}}_{L^q})\log \left((1+\|\dv u\|_{L^\infty})\|\na\rho\|_{L^q}+\|\rho\dot{u}\|_{L^q}\right)+C\|\nabla u\|_{L^2}+C\\
&\leq C(1+\|\rho\dot{u}\|^{\f{q}{2(q-1)}}_{L^q})\log \left(1+\|\na\rho\|_{L^q}+\|\rho\dot{u}\|_{L^q}\right)+C\|\nabla u\|_{L^2}+C\\
&\leq C(1+\|\rho\dot{u}\|_{L^q})\log \left(1+\|\na\rho\|_{L^q}\right)+C\|\nabla u\|_{L^2}+C.
\end{split}\eeo

On the other hand, by Lemma \ref{interpolationlem2}, we have that
\beno
\|\dot{u}\|_{L^2} \leq C( \|\rho^{\f12}\dot{u}\|_{L^2}+ \|\nabla \dot{u}\|_{L^{2}}),
\eeno
which implies that, by Proposition \ref{highest2},
\beno
\|\rho\dot{u}\|_{L^2(0, T; L^q)}\leq C.
\eeno

Pluging above estimate into \eqref{Phi} and applying the Gronwall's inequality, we obtain
\beno
\|\nabla \rho\|_{L^q}\leq C.
\eeno
Thus, we have 
\beno
\|\nabla^2 u\|_{L^q}\leq C(\|\na \dv u\|_{L^q}+\|\na \om\|_{L^q})\leq C(1+\|\nabla\rho\|_{L^q}+\|\rho\dot{u}\|_{L^q}),
\eeno
which implies that
\beno
\int^T_0 \|\nabla^2 u\|^2_{L^q}\,dt+\|\nabla^2 u\|_{L^2} \leq C,
\eeno
which completes the proof of Proposition \ref{highest2}.
\end{proof}

\medskip

{\bf Proof of the Proposition \ref{Blow-upC}:} 
Now we are in the position to prove Proposition \ref{Blow-upC}.
We prove it by the contradiction argument. Assume that $T^*<\infty$ and
$$\sup\limits_{s\in [0,T^*)}\|\rho(s)\|_{L^{\infty}((0, T^*)\times\Om)}<\infty.$$
Then, by Proposition \eqref{highest2}, we have
\beo
\sup_{0\leq t\leq T^*} (\|\rho(t)\|_{ W^{1, q}}+\|\nabla u(t)\|^2_{H^1} ) +\int_0^{T^*}  \|\nabla^2 u\|^2_{L^q} \,dt\leq C.
\eeo
By Lemma \ref{locallem}, we can extend the solution to $[0, T^*+\e]$ for some small data $\e$. Thus, the Proposition \ref{Blow-upC} is proved.

\medskip

{\bf Proof of the Theorem \ref{mainthm1}:}  We prove it by the contradiction argument. Assume that lifespan $T^*<\infty$. Thus, by the Proposition \ref{densityprop}, we have
\beno
\sup\limits_{s\in [0,T^*)}\|\rho(s)\|_{L^{\infty}((0, T^*)\times\Om)}<\infty.
\eeno
Then, applying Proposition \ref{Blow-upC}, we can extend the solution which contradicts with the definition of the $T^*$.


\begin{thebibliography}{50}
\bibitem{BKM} Beale, J.~T.; Kato, T.; Majda, A. Remarks on the breakdown of smooth solutions for the 3-D Euler equations. {\it Commun. Math. Phys.} {\bf 94} (1984), 61-66.

\bibitem{BW} Br$\acute{e}$zis, H.; Wainger, S. A note on limiting cases of Sobolev embeddings and convolution inequalities. {\it Comm. Partial Diff. Eq.} {\bf 5} (1980), no.7, 773-789.

\bibitem{Charve} Charve, F.; Danchin, R.
A global existence result for the compressible Navier-Stokes
equations in the critical $L^p$ framework. {\it Arch. Rational Mech.
Anal.} {\bf 198} (2010), 233-271.

\bibitem{CMZ-Rev} Chen, Q.; Miao, C.; Zhang, Z. Well-posedness in critical spaces for
the compressible Navier-Stokes equations with density dependent viscosities.
{\it Revista Mat. Iber.} {\bf 26} (2010), 915-946.

\bibitem{CMZ-CPAM} Chen, Q.; Miao, C.; Zhang, Z. Global well-posedness for compressible Navier-Stokes
equations with highly oscillating initial velocity. {\it Comm. Pure Appl. Math.} {\bf 63} (2010), 1173-1224.

%\bibitem{CMZ-Arxiv} Chen, Q.; Miao, C.; Zhang, Z. On the ill-posedness of the compressible Navier-Stokes equations.
%http://arxiv.org/abs/1109.6092.

\bibitem{CLMS} Coifman, R.; Lions, P.~L.; Meyer, Y.; Semmes, S. Compensated-Compactness and Hardy spaces, {\it J. Math. Pure Appl.}, {\bf 72} (1993), 247--286.

\bibitem{CM} Coifman, R.; Meyer, Y. On commutators of singular intergrals and bibinear singular integrals. {\it Trans. Amer. Math. Soc.} {\bf 212} (1975), 315-331.

\bibitem{CRW} Coifman, R.; Rochberg, R.; Weiss, G. Factorization theorems for Hardy spaces in several variables. {\it Ann. of Math.} {\bf 103} (1976), 611-635.

\bibitem{Dan-Inve} Danchin, R.
Global existence  in  critical spaces for compressible Navier-Stokes equations.
{\it Invent. Math.} {\bf 141} (2000), 579-614.

\bibitem{Desj} Desjardins, B. Regularity results for two-dimensional flows of multiphase viscous fluids,
 {\it Arch. Rat. Mech. Anal.} {\bf 137} (1997), 135-158.

\bibitem{Eng} Engler, H. An alternative proof of the Brezis-Wainger inequality, {\it Commun. Partial Differential Equations} {\bf 14} (1989), 541-544.

\bibitem{FZ} Fang, D.; Zhang, T. Compressible flows with a density-dependent viscosity coefficient. {\it SIAM J. Math. Anal.} {\bf 41} (2009), no.6, 2453-2488.

\bibitem{Feff} Fefferman, C. Characterizations of bounded mean oscillation. {\it Bull. Amer. Math. Soc.} {\bf 77} (1971), 587-588.

\bibitem{Fei} Feireisl, E. {\it Dynamics of viscous compressible fluids.} Oxford University Press, 2004.

\bibitem{FNP} Feireisl, E.; Novotny, A.; Petzeltov$\acute{a}$, H. On the existence of globally defined weak solutions to the Navier-Stokes equations. {\it J. Math. Fluid Mech.} {\bf 3} (2001), no. 4, 358-392.

\bibitem{Hoff-JDE} Hoff, D. Global solutions of the Navier-Stokes equations for multidimensional compressible
flow with discontinuous initial data. {\it J. Differ. Eqs.} {\bf 120} (1995), 215-254.

%\bibitem{Hoff-ARMA} Hoff, D.  Discontinuous solutions of the Navier-Stokes equations
%for multidimensional flows of heat-conducting fluids. {\it Arch. Rational Mech. Anal.} {\bf 139} (1997), 303--354.
%
%\bibitem{Hoff-CPAM} Hoff, D.  Dynamics of singularity surfaces for compressible, viscous flows in two space dimensions.
%{\it Comm. Pure Appl. Math.}  {\bf 55} (2002), 1365-1407.

\bibitem{Hoff-JMFM} Hoff, D.  Compressible flow in a half-space with Navier boundary condtions.
{\it J. Math. Fluid Mech.} {\bf 7} (2005), 315-338.

\bibitem{HL} Huang, X.; Li, J. Existence and blowup behavior of global strong solutions to the two-dimensional baratropic compressible Navier-Stokes system with vacuum and large initial data, http://arxiv.org/abs/1205.5342.

\bibitem{HL1} Huang, X.; Li, J. Global Well-Posedness of Classical Solutions to the Cauchy problem of Two-Dimensional Baratropic Compressible Navier-Stokes System with Vacuum and Large Initial Data, preprint.

\bibitem{Huang} Huang, X.; Li, J.; Xin, Z. Global well-posedness of classical solutions with large oscillations
and vacuum to the three dimensional isentropic compressible Navier-Stokes equations. http://arxiv.org/abs/1004.4749.

\bibitem{Jiang1} Jiang, S.; Zhang, P. Global spherically symmetric solutions of the compressible isentropic
Navier-Stokes equations. {\it Comm. Math. Phys.} {\bf 215} (2001), 559-581.

\bibitem{Jiang2} Jiang, S.; Zhang, P. Axisymmetric solutions of the 3-D
Navier-Stokes equations for compressible isentropic flows. {\it J. Math. Pure Appl.} {\bf 82} (2003), 949-973.

\bibitem{LSU} Ladyzenskaja, O.~A.; Solonnikov, V.~A.; Ural'ceva, N.~N. {\it Linear and quasilinear equations of parabolic type.} American Mathematical Society, Providence, RI 1968.

\bibitem{Lions} Lions, P.~L. {\it Mathematical Topics in Fluid Mechanics}. Vol.{\bf 2}, Compressible models.
Oxford University Press, 1998.

%\bibitem{Mat}  Matsumura, A.; Nishida, T. The initial value problem for the equations of motion of compressible viscous and heat-conductive fluids.
%{\it Proc. Japan Acad. Ser. A Math. Sci.} {\bf 55} (1979), 337-342.

\bibitem{MN} Matsumura, A.; Nishida, T. The initial value problem for the equations of motion of viscous and heat-conductive gases. {\it J. Math. Kyoto Univ.} {\bf 20} (1980), no. 1, 67-104.

\bibitem{Nash} Nash, J. Le probleme de Cauchy pour les equations differentielles d'un fluide general. {\it Bull. Soc. Math. France.} {\bf 90} (1962), 487-497.

%\bibitem{Niren} Nirenberg, L. On elliptic partial differential equations. {\it Ann. Scuola Norm. Sup. Pisa} (3) {\bf 13} (1959), 115-162.

\bibitem{Pere} Perepelitsa, M. On the global existence of weak solutions for the Navier-Stokes equations of compressible fluid flows. {\it SIAM. J. Math. Anal.} {\bf 38} (2006), no.1, 1126-1153.

\bibitem{SS} Salvi, R.; Stra$\check{s}$kraba, I. Global existence for viscous compressible fluids and their behavior as $t\to \infty.$ {\it J. Fac. Sci. Univ. Tokyo Sect. IA Math.} {\bf 40} (1993), no.1, 17-51.

%\bibitem{Serrin} Serrin, J. On the uniqueness of compressible fluid motion. {\it Arch. Rational. Mech. Anal.} {\bf 3} (1959), 271-288.

\bibitem{Solon} Solonnikov, V.~A. Solvability of the initial-boundary-value problem for the equation of a viscous compressible fluid. {\it J. Math. Sci.} {\bf 14} (1980), 1120-1133.

\bibitem{Sun} Sun, Y.; Wang, C.; Zhang, Z. A Beale-Kato-Majda blow-up criterion for the 3-D compressible
Navier-Stokes equations. {\it J. Math. Pure Appl.} {\bf 95} (2011), 36-47.

\bibitem{Sun-ARMA}  Sun, Y.; Wang, C.; Zhang, Z. A Beale-Kato-Majda criterion for three dimensional compressible viscous heat-conductive flows.
 {\it Arch. Ration. Mech. Anal.} {\bf 201} (2011), 727-742.

\bibitem{VK} Vaigant, V.~A.; Kazhikhov, A.~V. On existence of global solutions to the two-dimesional Navier-Stokes equations for a compressible viscous fluid. {\it Sib. Math. J.} {\bf 36} (1995), no.6, 1283-1316.

\bibitem{WWZ} Wang, C.; Wang, W.; Zhang, Z. Global well-posedness of compressible Navier-Stokes equations for some classes of large initial data.
to appear in {\it Arch. Ration. Mech. Anal.}.

\bibitem{Xin} Xin, Z. Blowup of smooth solutions to the compressible Navier-Stokes equation with compact density.
{\it Comm. Pure Appl. Math.} {\bf 51} (1998), 229-240.


%\bibitem{Bah} H. Bahouri, J.-Y. Chemin and R. Danchin,
%{\it Fourier analysis and nonlinear partial differential equations},
%Fundamental Principles of Mathematical Sciences, 343, Springer, Heidelberg, 2011.
%
%\bibitem{Bony} J.-M. Bony, {\it Calcul symbolique et propagation
%des singulariti\'{e}s pour les \'{e}quations aux d\'{e}riv\'{e}es
%partielles non lin\'{e}aires}, Ann. de l'Ecole Norm. Sup.,
%14(1981), 209-246.
%
%\bibitem{Can} M. Cannone,
%{\it A generalization of a theorem by Kato on Naiver-Stokes equations},
%Revista Mat. Iber., 13 (1997), 515-541.
%
%\bibitem{CMP} M. Cannone, Y. Meyer and F. Planchon,
%{\it Solutions autosimilaires des {\'e}quations de Navier-Stokes},
%{S{\'e}minaire ``{\'E}quations aux D{\'e}riv{\'e}es Partielles" de l'{\'E}cole polytechnique},
%Expos{\'e} VIII, 1993-1994.










%\bibitem{Dan-CPDE} Danchin, R.
%Local theory in critical spaces for compressible viscous and heat-conductive gases.
%{\it Comm. Partial Differential Equations} {\bf 26} (2001), 1183-1233.
%
%\bibitem{Dan-ARMA}  Danchin, R.
%Global existence in critical spaces for flows of compressible viscous and heat-conductive gases.
% {\it Arch. Rational Mech. Anal.} {\bf 160} (2001), 1-39.
%
%\bibitem{Dan-Non} Danchin, R. On the uniqueness  in critical spaces for compressible Navier-Stokes equations.
%{\it Nonlinear Differential Equations Appl.} {\bf 12} (2005), 111-128.
%
%
%\bibitem{Dan-CPDE07} Danchin, R. Well-posedness in critical spaces for barotropic viscous fluids with truly not constant density.
%{\it Comm. Partial Differential Equations} {\bf 32} (2007), 1373-1397.























\end{thebibliography}
\end{document}